%% file: HopfOre-current.tex
\documentclass[11pt]{amsart}

\include{preamble-HopfOre}

\usepackage{enumerate}

\newtheorem*{theorem*}{Theorem}
\theoremstyle{remark}

\title{Actions of Hopf-Ore Extensions of Group Algebras on Path Algebras of Quivers}

\author{Elise Askelsen}
\address{University of Iowa, Department of Mathematics, Iowa City, USA}
\email[Elise Askelsen]{elise-askelsen@uiowa.edu}
\thanks{EA acknowledges support from the Erwin and Peggy Kleinfeld Graduate Fellowship fund.}

\author{Ryan Kinser}
\address{University of Iowa, Department of Mathematics, Iowa City, USA}
\email[Ryan Kinser]{ryan-kinser@uiowa.edu}
\thanks{This work was supported by a grant from the Simons Foundation (636534, RK).  This material is based upon work supported by the National Science Foundation under Award No. DMS-2303334.}



\begin{document}

\begin{abstract}
    We classify the (filtered) Hopf actions of Hopf-Ore extensions of group algebras on path algebras of quivers, extending results in several other works from special cases to this general setting. Having done this for general Hopf-Ore extensions of group algebras, we demonstrate application by specializing our main result to certain Hopf-Ore extensions including some Noetherian prime Hopf algebras of GK-dimension one and two.
\end{abstract}

\maketitle

\setcounter{tocdepth}{1}
\tableofcontents

\section{Introduction}
\subsection{Context and Motivation}
    It has long been known that the classical notion of symmetry can be mathematically formalized using the language of group actions.
    Recently, there has been increasing interest in the more general notion of quantum symmetries, and one way to mathematically formalize this is in the language of actions of Hopf algebras (i.e. Hopf actions).
    Doing so generalizes the classical setting since an action of a group $G$ on a vector space $V$ is exactly the same thing as a Hopf action of the group algebra $\fk G$ on $V$.
    Some recent work on Hopf actions and their properties includes, but is not limited to these works
    \cite{CEW16, CKWZ16, EW16, CuaEt17, EW17, CenYas20, CliGad20, LNY20, BahMon21, KO21, EtNeg22, GadWon22, CWZ23}.
    We note that while the work of the current paper is entirely in the language of Hopf actions, our results can also be framed in the setting of tensor categories following \cite{EKW}.
    The base field $\fk$ is arbitrary except when otherwise noted throughout the paper.

    Of particular importance in the finite-dimensional setting are tensor algebras over semisimple algebras, and the special case of path algebras of quivers. These are the appropriate free objects which serve a similar role in the theory of associative algebras to what polynomial rings over fields serve in the setting of affine commutative algebras.  More precisely, any finite dimensional algebra over an algebraically closed field $\fk$ is isomorphic to a quotient of a bimodule tensor algebra over a product of matrix algebras over $\fk$, and any such algebra which is also basic is isomorphic to quotient of a path algebra of a quiver over $\fk$.  Therefore, any Hopf action on a finite-dimensional algebra descends from one on such a tensor algebra, or a path algebra of a quiver when the algebra is basic, motivating our study of Hopf actions on path algebras of quivers as a starting point to the much more general situation.

    Our work extends previous results about (filtered) Hopf actions on path algebras of quivers to the setting of Hopf-Ore extensions \cite{Pan03} of group algebras, a construction which encompasses many previous results about Hopf actions on path algebras as special cases.
    It also covers new notable examples, including several classes of Noetherian Hopf algebras of Gelfand-Kirillov dimension one and two in the classification by Brown-Zhang \cite{BZ10} and Goodearl-Zhang \cite{GZ10}.
    In more generality, it is known that any finite-dimensional, rank one, pointed Hopf algebra over an algebraically closed field $\fk$ is isomorphic to a quotient of a Hopf-Ore extension of its coradical \cite{KropandRad, Schero}.
    
    It is also known that iterating Hopf-Ore extensions produces additional Hopf algebras of significant interest, as outlined in \cite{BOZZ15, LWC18, L21}.
    Moreover, \cite{Zhou} proved that every connected, graded Hopf algebra with finite GK-dimension over a field, $\fk $, of characteristic 0 is an iterated Hopf-Ore extension of $\fk$.
    Thus, our work is a starting point towards classification of Hopf actions of these larger families of algebras.
    
\subsection{Summary of Main Results}
    This paper focuses on actions of Hopf-Ore extensions of group algebras, $\fk G (\chi, h, \alpha)$, on path algebras of quivers, with essential background and notation related to these concepts reviewed in Sections \ref{section: background HA and Hact}, \ref{section: Quivers}, and \ref{section: HOE}.
    In Section \ref{section:action of x in H}, we record some preliminary results about actions of skew-primitive elements on a path algebra $\fk Q$ within the context of actions of an arbitrary Hopf algebra $H$.

    We prove our main result in Section \ref{section: actions of HOE of GA}, paraphrasing it as follows for this introduction (see Theorem \ref{prop:Action of R on kQ} for the full version). 

    \begin{thm}
    The following data determines a (filtered) Hopf action of the Hopf-Ore extension $R=\fk G(\chi, h, \delta)$ on $\fk Q$, and all such actions are of this form:
        \begin{enumerate}
            \item a Hopf action of $R$ on $\fk Q_0$ as explicitly described in Proposition \ref{prop:Action of R on Q_0};
            \item a representation of $G$ on $\fk Q_1$ compatible with the $\fk Q_0$-bimodule structure on $\fk Q_1$;
            \item a $\fk$-linear endomorphism of $\fk Q_0 \oplus \fk Q_1$ satisfying explicitly given relations depending on the data in the first items, and the structure of $R$.
       \end{enumerate}
    \end{thm}
    After proving our general classification result for these actions, we turn our attention to applying it to specific Noetherian Hopf algebras of GK-dimension one and two in Section \ref{section:NHAOGKD1&2}. 

\subsection{Relation to existing literature}\label{sec:lit}
Investigation of Hopf actions on path algebras of quivers began in \cite{KW16}, where the authors restricted their attention to Taft algebras and $u_q(\mathfrak{sl}_2)$ acting on path algebras of finite, loopless, and Schurian quivers.
The doctoral thesis of Berrizbeitia \cite{B18} furthered this by removing the conditions on the quiver, aside from finiteness, and also giving a description of the associated invariant rings.  

In \cite{KO21}, the authors extended the classification of Hopf actions on path algebras of quivers to $U_q(\mathfrak{b})$ where $\mathfrak{b} \subset \mathfrak{sl}_2$.  Concretely, 
fixing $q \in \fk$ with $q \neq \pm 1$, the Hopf algebra $U_q(\mathfrak{b})$ has $\fk$-algebra generators $x, g^{\pm 1}$ subject to the relations 
\begin{equation*}
        gg^{-1}=1=g^{-1}g \textrm{ and } xg=q^{-1}gx.
\end{equation*}
The comultiplication is given by $\Delta(g)=g \otimes g$ and $\Delta(x) = 1 \otimes x + x \otimes g$.  They also classify actions of $U_q(\mathfrak{sl}_2)$, determine when these actions descend to certain finite-dimensional quotients such as generalized Taft algebras and $u_q(\mathfrak{sl}_2)$, and determine the structure of bimodule categories over commutative, semisimple $H$-module algebras for $H$ being any of the above mentioned Hopf algebras.

Our extension to the general class of Hopf-Ore extensions of group algebras recovers some of these results as a special case.
In the notation of Section \ref{section: HOE}, $U_q(\mathfrak{b}):=\fk G(\chi , g, 0)$ where 
$G= \langle g^{\pm 1} \rangle$ and 
$\chi(g)=q^{-1}$,
so our Theorem \ref{prop:Action of R on kQ} immediately implies Proposition 3.1 and Theorem 3.13 of \cite{KO21}.
Future extension of this work to iterated Hopf-Ore extensions should allow us obtain the classification results for actions of $U_q(\mathfrak{sl}_2)$ as a special case.
We leave bimodule categories over Hopf-Ore extensions for future investigation.

\section{Background} \label{section: background}

\subsection{Hopf Algebras and Hopf Actions}\label{section: background HA and Hact}
In this section, we outline some of the basic definitions and background on Hopf algebras and Hopf actions here, referring the reader to standard texts such as \cite{Mont, Rad} for more detail.
Throughout this paper, the term \emph{algebra} always refers to an associative $\fk$-algebra with identity where $\fk$ is a field.
We use the notation $\Delta, \varepsilon,$ and $S$ for the comultiplication, counit, and antipode of a Hopf algebra, respectively and employ Sweedler notation with suppressed summation, i.e. $\Delta(a)=a_1 \otimes a_2$. 

The two main types of elements in a Hopf algebra that we will concern ourselves with in this paper are the grouplike and skew-primitive elements.  We denote the subset of \emph{grouplike elements} by 
\[
G(H) = \left\{ g \in H \mid \Delta(g) = g\otimes g \right\},
\]
which is a subgroup of the multiplicative group of units in the algebra.
Using Radford's convention, an element $x \in H$ is called \emph{$(g,h)$-skew primitive} if $\Delta(x)=g \otimes x + x \otimes h $ where $g,h \in G(H)$ \cite{Rad}.
In the case that $g=h=1$, the element $x$ is called a \emph{primitive element}.

\begin{remark} \label{rmk:structure of comultiplication}
Left multiplication restricts to an action of $G(H)$ on the set of skew-primitive elements of $H$, with $k \in G(H)$ sending a $(g,h)$-skew primitive element to a $(kg,kh)$-skew primitive element.
Because of this, when dealing with skew-primitives it often suffices to work with $(1, h)$-skew primitive elements.
\end{remark}

With this in mind, we can define the main topic of study in this work.

\begin{definition}\label{def:Hopfaction}
    Let $H$ be a Hopf algebra and $A$ be an algebra. 
    A \emph{(left) Hopf action} of $H$ on $A$ consists of a left $H$-module structure on $A$ satisfying:
    \begin{enumerate}
        \item $h \cdot (pq)=(h_{1} \cdot p)(h_{2} \cdot q)$ for all $h \in H$ and $p,q \in A$, and
        \item $h \cdot 1_A= \varepsilon (h) 1_A$ for all $h \in H$. 
    \end{enumerate}
    In this case, we also say that $A$ is a \emph{left $H$-module algebra}.
\end{definition}

\subsection{Quivers and Path Algebras} \label{section: Quivers}{
In this section, we introduce the algebras on which the Hopf algebras in this paper act.
For a more in-depth discussion of quivers and their path algebras, we recommend sources such as \cite{Schiffler14, derksen}. Recall that
a \emph{quiver}, $Q=(Q_0, Q_1, s,t)$, consists of a set of vertices $Q_0$, a set of arrows $Q_1$, and two maps $s,t \colon Q_1 \longrightarrow Q_0$ sending an arrow to its source and target respectively.
We sometimes omit the parentheses for the source and target maps to improve readability, writing $sa$ and $ta$ for $a \in Q_1$.
In the context of this paper, we require both $Q_0$ and $Q_1$ to be finite sets.

By concatenating arrows in the quiver, paths can be formed, leading to an algebraic object, namely the path algebra associated to $Q$.

\begin{definition}\label{def:path algebra}
    The \emph{path algebra} $\fk Q$ of the quiver $Q$ is the algebra with basis the set of all paths in $Q$ and multiplication defined on two basis elements $p,p'$ as 
    \begin{equation*}
    pp'= \begin{cases} 
      p \cdot p' & \textrm{if } s(p')=t(p) \\
      0 & \textrm{otherwise}.
   \end{cases}
   \end{equation*}
\end{definition}

We note that this includes a path of length zero $e_i \in \fk Q$ for all $i \in Q_0$, and paths of greater length are given by nonzero products of arrows.
With this, it is clear that $\fk Q$ has the direct sum decomposition as a vector space 
\begin{equation*}
    \fk Q = \bigoplus_{l=0 }^ \infty \fk Q_l
\end{equation*}
where $\fk Q_l$ is the subspace of $\fk Q$ with basis the set of all paths of length $l$, denoted $Q_l$, and this
makes $\fk Q$ a graded $\fk$-algebra. 

We also note that $\fk Q$ can be viewed as a tensor algebra: $\fk Q_0=\bigoplus_{i \in Q_0} \fk e_i$ is a semisimple $\fk$-algebra, while $\fk Q_1$ is a $\fk Q_0$-bimodule via the relations of the path algebra
\begin{equation}\label{eq:path algebra relations}
e_ia=\delta_{i,sa}a \qquad \text{and} \qquad ae_i=\delta_{i,ta}a
\end{equation}
for all $i \in Q_0$ and $a \in Q_1$. 
These facts induce an isomorphism $\fk Q \cong T_{\fk Q_0}(\fk Q_1)$ where $T_{\fk Q_0}(\fk Q_1)$ is the tensor algebra of $\fk Q_1$ over $\fk Q_0$.
    
Following earlier works, we restrict our attention to Hopf actions on $\fk Q$ which are \emph{filtered} in the following sense.

\begin{assumption}\label{def:assumption}
    We assume that Hopf actions on path algebras in this paper preserve the ascending filtration by path length:
    \begin{equation*}
    \fk Q_0 \subset \fk Q_0 \oplus \fk Q_1 \subset \fk Q_0 \oplus \fk Q_1 \oplus \fk Q_2 \subset \cdots.
    \end{equation*}
\end{assumption}

This assumption leads to the following lemma (see \cite[Lem. 2.27]{KO21}) which justifies our restriction of attention to vertices and arrows when studying Hopf actions on path algebras.
\begin{lemma}\label{lem: ko21 Lemma 2.27}
Let $H$ be a Hopf algebra and $Q$ a quiver.
    \begin{enumerate}
        \item Suppose a Hopf action of $H$ on $\fk Q$ is given. Then the action is completely determined by the $H$-module structure of $\fk Q_0 \oplus \fk Q_1$.
        \item Conversely, any $H$-module structure on $\fk Q_0 \oplus \fk Q_1$ 
        which preserves the relations of the algebra $\fk Q_0$ and the bimodule relations in \eqref{eq:path algebra relations}
        uniquely extends to a Hopf action on $\fk Q \cong T_{\fk Q_0}(\fk Q_1)$.
        \item Any action of $G(H)$ on $\fk Q_0$ is induced by a permutation action of $G(H)$ on $Q_0$.
    \end{enumerate}
\end{lemma}

\subsection{Preliminary results on actions of skew-primitive elements}
\label{section:action of x in H}
In this section, we establish some preliminary results about the actions of skew primitive elements on path algebras of quivers.
The ideas and proofs are minor generalizations of earlier works, but we include them for completeness.

Throughout this section, we fix an $(1,h)$-skew primitive element $x \in H$,
where $H$ is a Hopf algebra with an arbitrary (filtered) action on $\fk Q$.
As in Lemma \ref{lem: ko21 Lemma 2.27}, this action is completely determined by the $H$-module structure of $\fk Q_0 \oplus \fk Q_1$, which in particular induces a permutation action of $G(H)$ on $Q_0$ and a linear representation of $G(H)$ on $\fk Q_1$.
We first determine how $x$ may act on $\fk Q_0$.  
The following proposition follows essentially the same proof as \cite[Prop.~3.1(b)]{KO21}.

\begin{proposition}\label{prop:form of x action}
    Under the assumptions of this section, there exists a collection of scalars $(\gamma_i \in \fk)_{i \in Q_0}$ such that for all $i \in Q_0$
    \[
    x \cdot e_i= \gamma_ie_{ i}- \gamma_{h \cdot i}e_{h \cdot i}. \label{eq:general action of x on e_i}
    \]
    
\end{proposition}
\begin{proof}
Because of the assumption that the action is filtered, $x\cdot e_i \in \fk Q_0$.
Using that $e_i^2=e_i$ and the definition of a Hopf action, 
\begin{equation}
    x \cdot e_i = x \cdot (e_ie_i)=e_i (x \cdot e_i) + (x \cdot e_i) (h \cdot e_i) 
    = e_i ( x \cdot e_i) + (x \cdot e_i) e_{h \cdot i} .\label{eq:defn of action on x e_i}
\end{equation}
This shows that $x\cdot e_i$ must be a linear combination of the elements $e_{h \cdot i}$ and $e_i$ for all $i \in Q_0$ and thus there exists some collection of scalars, $(\gamma_i, \gamma_i' \in \fk)_{i \in Q_0}$, say 
\begin{equation}\label{eq:xdotei2}
    x \cdot e_i = \gamma_ie_{ i} + \gamma'_i e_{h \cdot i}.
\end{equation}

First consider the special case where $h\cdot i = i$. Plugging the expression \eqref{eq:xdotei2} for $x\cdot e_i$ into all occurrences on both sides of \eqref{eq:defn of action on x e_i} gives
\begin{equation}
\begin{split}
     \gamma_ie_{ i} +\gamma'_ie_{h \cdot i}&=e_i(\gamma_ie_{ i} +\gamma'_ie_{h \cdot i})+(\gamma_ie_{ i} +\gamma'_ie_{h \cdot i})e_{h \cdot i}\\
    &=\gamma_ie_{ i}+\gamma'_ie_{ i}e_{h \cdot i}+\gamma_ie_{h \cdot i}e_{ i}+\gamma'_ie_{h \cdot i}.
    \end{split}
\end{equation}
Canceling like terms from both sides yields $0 =(\gamma_i+\gamma'_i)e_{ i}e_{h \cdot i}$.
Since we assume that $h \cdot i = i$ in this case, then $\gamma_i'=-\gamma_i$ and thus $x \cdot e_i = 0$.
Therefore the statement of the proposition is true with $\gamma_i = \gamma_{h\cdot i}=0$.

Now consider the case where $h \cdot i \neq i$. 
Then $e_ie_{h\cdot i}=e_{h\cdot i}e_{h^2\cdot i}=0$, so we get
\begin{equation}
\begin{split}
    0 & = x \cdot (e_{h \cdot i}e_i) = e_{h \cdot i}(x \cdot e_i) + (x \cdot e_{h \cdot i})e_{h \cdot i}\\
    &= e_{h \cdot i}(\gamma_ie_{i}+\gamma'_ie_{h \cdot i}) + (\gamma_{h\cdot i}e_{h\cdot i}+\gamma'_{h\cdot i}e_{h^2\cdot i})e_{h \cdot i}\\
    & = (\gamma'_i+\gamma_{h\cdot i})e_{h\cdot i}
\end{split}
\end{equation}
from which it follows that $\gamma'_i=-\gamma_{h \cdot i}$, completing the proof.
\end{proof}

We extract the following lemma for reference later, which was proven in the second paragraph of the above proof.
\begin{lemma}\label{lemma:stab}
If $i \in Q_0$ satisfies $h \cdot i = i$, then $x \cdot e_i=0$.
\end{lemma}

Next, we turn our attention to the action of $x$ on $\fk Q_1$.

\begin{proposition} \label{prop:H action on kQ}
    Under the assumptions of this section, there exists a $\fk$-linear endomorphism $\sigma: \fk Q_0 \oplus \fk Q_1 \to \fk Q_0 \oplus \fk Q_1$ such that
    \begin{equation}\label{eq:xdotaform}
        x \cdot a = \gamma_{ta}a-\gamma_{h \cdot sa}(h \cdot a) + \sigma(a),
    \end{equation}
    for all $a \in Q_1$.  Furthermore, $\sigma$ has the properties 
    $\sigma(\fk Q_0)=0$ and $\sigma(a)=e_{sa}\sigma(a)e_{h \cdot ta}$ for all $a \in Q_1$.
\end{proposition}
\begin{proof}
    Set $\sigma(\fk Q_0)=0$ and define $\sigma(a)=e_{sa}(x \cdot a)e_{h \cdot ta}$ for all $a \in Q_1$.
    From this definition it is obvious that $\sigma(a)=e_{sa}\sigma(a)e_{h \cdot ta}$ since $e_i$ is an idempotent element of $\fk Q$ for all $i \in Q_0$.

    Using that $a = e_{sa} a e_{ta}$ in $\fk Q$, that $\Delta^2(x) =1\otimes 1\otimes x +  1 \otimes x \otimes h + x\otimes h \otimes h $, and Proposition \ref{prop:form of x action}, we find
    \begin{equation}
    \begin{split}
        x \cdot a= x \cdot (e_{sa}ae_{ta})
        & = e_{sa}a(x \cdot e_{ta}) + e_{sa}(x \cdot a)e_{h \cdot ta} + (x \cdot e_{sa})(h \cdot a)e_{h \cdot ta}\\
        &= a (\gamma_{ta}e_{ta}-\gamma_{h \cdot ta}e_{h \cdot ta}) + \sigma(a) + (\gamma_{sa}e_{sa}-\gamma_{h \cdot sa}e_{h \cdot sa})(h \cdot a)\\
        & = \gamma_{ta}a-\gamma_{h \cdot ta}ae_{h \cdot ta} + \sigma(a) + \gamma_{sa}e_{sa}(h \cdot a)-\gamma_{h \cdot sa}(h \cdot a)
    \end{split}
    \end{equation}

There are four cases for how the final expression simplifies, based on whether $h$ fixes $sa$ and $ta$.  Recall from Lemma \ref{lemma:stab} that $h\cdot i = i$ implies that $\gamma_i=0$.

    \textbf{Case 1:} Suppose $sa = h \cdot sa$ and $ta = h \cdot ta $. Then $\gamma_{sa} = \gamma_{h\cdot sa} = \gamma_{ta} = \gamma_{h\cdot ta} = 0$, so $x \cdot a = \sigma(a)$.

    \textbf{Case 2:} Suppose that $sa = h \cdot sa$ and $ta \neq h \cdot ta$. This means $\gamma_{sa} = \gamma_{h\cdot sa}=0$ and $ae_{h\cdot ta}=0$.
    Therefore $x \cdot a = \gamma_{ta} a+ \sigma(a)$.

    \textbf{Case 3:} Suppose $sa \neq h \cdot sa$ and $ta = h \cdot ta $. Hence $\gamma_{ta} = \gamma_{h\cdot ta} = 0$ and $e_{sa}(h \cdot a) = 0$, so $x \cdot a = -\gamma_{h \cdot sa} ( h \cdot a)+ \sigma(a)$.

    \textbf{Case 4:} Suppose $sa \neq h \cdot sa$ and $ta \neq h \cdot ta$. Then $ae_{h\cdot ta}=0=e_{sa}(h \cdot a)$, so
    $x \cdot a = \gamma_{ta}a-\gamma_{h \cdot sa} ( h \cdot a)+ \sigma(a)$.

    In each possible case, the action of $x$ on $a \in \fk Q_1$ is of the form in \eqref{eq:xdotaform}, completing the proof.
\end{proof}

 Additional conditions on the coefficients $(\gamma_i \in \fk)_{i \in Q_0}$ arise depending on additional relations between $x$ and other elements of $H$. 
 In this paper, we focus on the case when $x$ is a variable adjoined to an algebra to form an Ore-extension (i.e. skew-polynomial ring).

\subsection{Hopf-Ore Extensions of Group Algebras} \label{section: HOE}
This section reviews and establishes notation for Hopf-Ore extensions following work of Panov \cite{Pan03}. 
We narrow our focus to Hopf-Ore extensions of group algebras, 
motivated by the fact that many Hopf algebras of interest arise in this way, even though Panov's results are in more generality.
We fix $A$ to be a $\fk$-algebra, noting that in Panov's general definition of a Hopf-Ore extension, $\fk$ can be replaced by any commutative ring with 1.
To begin, we recall the data defining Ore extensions.

Let $\tau$ be a $\fk$-linear endomorphism of $A$.
A \emph{$\tau$-derivation} of $A$ is a $\fk$-linear map $\delta\colon A \longrightarrow A$ such that $\delta(\fk)=0$ and $\delta(ab)= \tau(a) \delta(b)+ \delta(a)b$ for all $a,b \in A$.

The \emph{Ore extension} $R=A[x; \tau, \delta]$ of $A$,
also commonly referred to as a \emph{skew polynomial ring},  is the $\fk$-algebra $R$ generated by the variable $x$ and the algebra $A$ with the relation that for all $a \in A$
\begin{equation}
    xa=\tau(a)x+\delta(a).
\end{equation}

Panov defines a Hopf algebra structure on an Ore extension of a Hopf algebra in the following way.

\begin{definition}\cite[Def.~1.0]{Pan03}
    Let $A$ be a Hopf algebra, and suppose $R=A[x; \tau, \delta]$ is also a Hopf algebra such that $A\subset R$ is a Hopf subalgebra and $\Delta(x) = r_1 \otimes x + x \otimes r_2$ for some $r_1, r_2 \in G(A)$.
    Then $R$ is called a \emph{Hopf-Ore extension} of $A$.
    
\end{definition}

As mentioned in Remark \ref{rmk:structure of comultiplication}, it will be assumed throughout this paper that $x$ is $(1,h)$-skew primitive (i.e. $\Delta(x)= 1 \otimes x + x \otimes h$).
Panov proved that when $A$ is cocommutative, $h$ is in the center of $A$ \cite[Cor.~1.4]{Pan03}, thus in the case $A=\fk G$ we have that $h$ is in the center of $G$.  In fact, Panov gave a complete classification of Hopf-Ore extensions of group algebras with only two additional pieces of information \cite[Prop. 2.2]{Pan03}.
The first is a character $\chi \colon G \to \fk^\times$ such that for any $g \in G$, $\tau(g)=\chi(g)g$ while the second is a linear form such that $\alpha(uv)=\alpha(u) + \chi(u)\alpha(v)$ for all $u,v \in G$.

\begin{proposition}\label{prop:HopfOre classification}
    Let $G$ be a group and $R=\fk G[x;\tau, \delta]$ a Hopf-Ore extension of $\fk G$.
    Then (up to a change of variable) we have for some $\chi, \alpha$ as above:
\begin{itemize}
    \item $\Delta(x)=1 \otimes x + x \otimes h$ for some $h$ in the center of $G$,
    \item $\tau(g)=\chi(g)g$ for all $g \in G$, and
    \item $\delta(g)= \alpha(g)(1-h)g$ for all $g \in G$.
\end{itemize}
In particular, the relation $xg=\chi(g)gx+ \alpha(g)(1-h)g$ holds for all $g \in G$.
\end{proposition}

We write $R=\fk G( \chi, h, \alpha)$ for a Hopf-Ore extension as described in Proposition \ref{prop:HopfOre classification}.

\section{Main results} \label{section: actions of HOE of GA}

We fix a quiver $Q$
and a Hopf-Ore extension of a group algebra $R=\fk G(\chi , h, \alpha)$
as in Section \ref{section: HOE}.
Building on the preliminary work in Section \ref{section:action of x in H} and utilizing Panov's results recalled above, we classify the actions of $R$ on $\fk Q$ in this section.  We consider the action of $R$ on $\fk Q_0$ first.

\begin{proposition}{\label{prop:Action of R on Q_0}}
        (a) The following data determines a Hopf action of $R$ on $\fk Q_0$. \label{condition 1.}
        \begin{enumerate}[(i)]
            \item A permutation action of $G$ on the set $Q_0$;
            \item A collection of scalars $(\gamma_i \in \fk )_{i \in Q_0}$ such that 
            \begin{equation}\label{eq: relation of R coefficients Q_0}
                \gamma_{g \cdot i}= \chi(g)\gamma_i + \alpha(g) \textrm{ for all }i \in Q_0  \textrm{ and for all } g \in G. 
            \end{equation}
            The $x$-action is given by 
            \begin{equation}
                x \cdot e_i = \gamma_ie_i -(\chi(h) \gamma_i + \alpha(h) )e_{h \cdot i}  \hspace{0.5cm} \textrm{ for all }i \in Q_0. \label{eq: R vertices Q_0}
            \end{equation}
        \end{enumerate}
        (b) Every action of $R$ on $\fk Q_0$ is of the form above.
\end{proposition}
\begin{proof}
        (a) It is straightforward to check that the data of (i) and (ii) gives $\fk Q_0$ an $R$-module structure. 
        The permutation action of $G$ on $Q_0$ induces a collection of algebra automorphisms of $\fk Q_0$, so it remains to 
        show that the action of $x$ in (\ref{eq: R vertices Q_0}) respects the relations $e_ie_j = \delta_{i,j}e_i$ of $\fk Q_0$ for all $i,j \in Q_0$. Using the definition of a Hopf action, this means
        \begin{equation}
            x\cdot \delta_{i,j}e_i = x \cdot (e_ie_j)= e_{ i}(x \cdot e_j) + (x \cdot e_i)e_{h \cdot j}. \label{eq:R on Q_0 initial definition fo Hopf action}
        \end{equation}
        must be verified.
        When $i=j$, plugging \eqref{eq: R vertices Q_0} into both sides of \eqref{eq:R on Q_0 initial definition fo Hopf action} yields
        \begin{equation}
        \begin{split}
            \gamma_ie_i -(\chi(h) \gamma_i + \alpha(h) )e_{h \cdot i}&= e_{i}(\gamma_ie_i -(\chi(h) \gamma_i + \alpha(h) )e_{h \cdot i})\\
            & \hspace{0.2in} +(\gamma_ie_i -(\chi(h) \gamma_i + \alpha(h) )e_{h \cdot i})e_{h \cdot i}\\
            & =\gamma_i e_{ i}- (\chi(h) \gamma_i + \alpha(h) )e_i e_{h \cdot i}\\
            & \hspace{0.2in} +\gamma_ie_{ i}e_{h \cdot i}- (\chi(h) \gamma_i + \alpha(h) )e_{h \cdot i}\label{eq:needed}
        \end{split}
        \end{equation}
        Canceling like terms from both sides gives $0=(\gamma_i-(\chi(h) \gamma_i + \alpha(h) ))e_i e_{h \cdot i}$.
        If $h \cdot i \neq i$, then $e_ie_{h \cdot i}=0$ so this is immediately satisfied.
        In the case that $h \cdot i = i$, the assumption \eqref{eq: relation of R coefficients Q_0}, along with the fact that $\gamma_i = \gamma_{h \cdot i}$, makes it clear that \eqref{eq:needed} is satisfied in this case as well.
        
        Next assume $i \neq j$. A similar expansion yields
        \begin{equation}
            \begin{split}
            0&=x \cdot (e_ie_j)=e_i(x \cdot e_j) + (x \cdot e_i)e_{h \cdot j}\\
            &=(\gamma_i - (\chi(h) \gamma_j + \alpha(h) ))e_ie_{h \cdot j} 
            \end{split}
        \end{equation}
        If $i \neq h \cdot j$, this is trivially satisfied because $e_ie_{h \cdot j}=0$. 
        If $i=h \cdot j$, this is satisfied from our assumption \eqref{eq: relation of R coefficients Q_0} since $\gamma_i=\gamma_{h \cdot j}=\chi(h) \gamma_j + \alpha(h)$.
        Since \eqref{eq:R on Q_0 initial definition fo Hopf action} holds in every case, this shows that the given data defines a Hopf action of $R$ on $\fk Q_0$.
        
        (b) For any Hopf action, the group of grouplike elements must act by $\fk$-algebra automorphisms, and the $\fk$-algebra automorphisms of $\fk Q_0$ are given by permutation actions on $Q_0$ \cite[Lem.~2.27(iii)]{KO21} as in (i).
        From Proposition \ref{prop:form of x action} there is a collection of scalars $(\gamma_i \in \fk)_{i \in Q_0}$ such that $x \cdot e_i = \gamma_ie_i - \gamma_{h \cdot i} e_{h \cdot i}$ for all $i \in Q_0$. 
        The requirement that in every Hopf action these satisfy \eqref{eq: relation of R coefficients Q_0} 
        comes from the relations $xg=\chi(g) gx + \alpha(g)(1-h)g$ in $R$, for all $g \in G$.
        To see this, we act on $e_i$ by both sides of this equation.
        Because $h$ is in the center of $R$, the left gives
        \begin{equation}
            (xg) \cdot e_i = x \cdot e_{g \cdot i} =\gamma_{g \cdot i} e_{g \cdot i} - \gamma_{gh \cdot i} e_{gh \cdot i}= g(\gamma_{g \cdot i} e_{i} - \gamma_{gh \cdot i} e_{h \cdot i}). \label{eq:left of relation}
        \end{equation}
        On the right, we obtain
\begin{equation}
        \begin{split}
            (\chi(g)gx + \alpha(g)g-\alpha(g)hg) \cdot e_i &= \chi(g)g(x \cdot e_i)+ \alpha(g)(g \cdot e_i) - \alpha(g) (hg \cdot e_i)\\
            &= g[\chi(g) (\gamma_i e_i - \gamma_{h \cdot i} e_{h \cdot i}) + \alpha(g) e_{i} -\alpha(g) e_{h \cdot i}]\\
             &= g[(\chi(g) \gamma_i + \alpha(g))e_{ i} - (\chi(g) \gamma_{h \cdot i} + \alpha(g))e_{h \cdot i}] \label{eq:right of relation}
        \end{split}
        \end{equation}
        Setting \eqref{eq:left of relation} and \eqref{eq:right of relation} equal to each other and canceling the $g$ on the left, we have
        \begin{align}
            \gamma_{g \cdot i} e_{i} - \gamma_{hg \cdot i} e_{h \cdot i}&=(\chi(g) \gamma_i + \alpha(g))e_{i} - (\chi(g) \gamma_{h \cdot i} + \alpha(g))e_{h \cdot i}\\
            0 &= (\chi(g) \gamma_i + \alpha(g)-\gamma_{g \cdot i})e_{i} - (\chi(g) \gamma_{h \cdot i} + \alpha(g)- \gamma_{hg \cdot i})e_{h \cdot i} \label{eq: Relation on R impacting action on Q_0}
        \end{align}
        
        If $h\cdot i = i$ then \eqref{eq: Relation on R impacting action on Q_0} is trivially satisfied and no additional relations are needed.
        However, if $h\cdot i \neq i$, the terms in \eqref{eq: Relation on R impacting action on Q_0} are linearly independent, so each of the coefficients must be zero and we obtain
        $\gamma_{g \cdot i}= \chi(g)\gamma_i + \alpha(g)$ and $\gamma_{gh \cdot i}= \chi(g) \gamma_{h \cdot i} + \alpha(g)$. 
        Since the first equation is quantified over all $g \in G$ and all $i \in Q_0$ such that $h\cdot i \neq i$, and the vertex $h\cdot i$ satisfies $h\cdot (h\cdot i) \neq h \cdot i$, the second condition is a consequence of the first.  This shows that \eqref{eq: relation of R coefficients Q_0} holds for every Hopf action of $R$ on $\fk Q_0$, and also that the $x$-action must be of the form \eqref{eq: R vertices Q_0}.
\end{proof}

As a corollary, we extend Lemma \ref{lemma:stab} by noting some conditions under which $x$ acts trivially on $e_{g \cdot i}$ for $g \in G$.

\begin{corollary} \label{cor: x dot ei =0 then x dot egi=0}
(a)    Suppose $x \cdot e_i =0$ and $h \cdot i \neq i$, and let $g \in G$. 
    Then $x \cdot e_{g \cdot i} =0 $ if and only if $\alpha(g)=0$.

(b) If $h \cdot i = i$, then $x \cdot e_{g \cdot i} = 0$ for all $g \in G$.
\end{corollary}
\begin{proof}
We note the following computation relevant to both parts:
\begin{equation}\label{eq:egi}
    \begin{split}
        x \cdot e_{g \cdot i}= (xg) \cdot e_i &= (\chi(g)gx+ \alpha(g)(1-h)g) \cdot e_i\\
        & = \chi(g)g(x \cdot e_i) + \alpha(g)g(e_i-e_{h\cdot i}).
        \end{split}
    \end{equation}

(a) From the assumption that $x \cdot e_i =0$ we get $x\cdot e_{g\cdot i}=\alpha(g)g(e_{i}- e_{h \cdot i})$.  Since we assume $ h\cdot i \neq i$, we have $e_i - e_{h\cdot i} \neq 0$ and the statement follows.

(b) Since $h \cdot i = i$ by assumption,  Lemma \ref{lemma:stab} implies that $x \cdot e_i=0$ and also the term $e_i-e_{h \cdot i}=0$ in \eqref{eq:egi} so the statement follows.
\end{proof}

We can now extend our result in Proposition \ref{prop:Action of R on Q_0} to the arrows of $Q$ in order to fully classify actions of $R$ on $\fk Q$.

\begin{theorem}\label{prop:Action of R on kQ}
    (a) The following data determines a Hopf action of $R$ on $\fk Q$.
        \begin{enumerate}[(i)]
            \item A Hopf action of $R$ on $\fk Q_0$;
            \item A representation of $G$ on $\fk Q_1$ satisfying $s(g \cdot a)=g \cdot sa$ and $t(g \cdot a)=g \cdot ta$ for all $a \in Q_1$ and all $g \in G$;
            \item A $\fk$-linear endomorphism $\sigma \colon \fk Q_0 \oplus \fk Q_1 \longrightarrow \fk Q_0 \oplus \fk Q_1$ satisfying
            \begin{enumerate}[($\sigma$1)]
                \item $\sigma(\fk Q_0)=0$;
                \item $\sigma(a)=e_{sa}\sigma(a) e_{h \cdot ta}$ for all $a \in Q_1$;
                \item $\sigma (g \cdot a)= \chi(g)g \sigma(a)+(e_{g \cdot sa}-e_{gh \cdot sa})\alpha(g)(g \cdot a) e_{hg \cdot ta}$ for all $a \in Q_1$ and all $g \in G$. \label{eq:relation}
            \end{enumerate}
            With this data, the $x$-action on $a \in Q_1$ is given by 
            \begin{equation}\label{eq:action of R on KQ}
                x \cdot a = \gamma_{ta}a-(\chi(h) \gamma_{sa} + \alpha(h) )(h \cdot a) + \sigma(a). 
            \end{equation}
        \end{enumerate}
    (b) Every (filtered) action of $R$ on $\fk Q$ is of the form above.
\end{theorem}
\begin{proof}
        (a) It is straightforward to check that the given data makes $\fk Q_0 \oplus \fk Q_1$ into a $R$-module as stated in Lemma \ref{lem: ko21 Lemma 2.27}.
        Because (ii) is equivalent to the given $G$-action on $\fk Q_1$ satisfying the definition of a Hopf action (i.e. $g \cdot a=g \cdot (e_{sa}a)=g \cdot (ae_{ta}$)), it suffices to show that the expression for the action of $x$ in \eqref{eq:action of R on KQ} satisfies the definition of a Hopf action, by showing it preserves the relations defining $\fk Q$.

        From the relations \eqref{eq:path algebra relations}, plugging \eqref{eq: R vertices Q_0} and \eqref{eq:action of R on KQ} into $e_{sa}(x \cdot a) + (x \cdot e_{sa}) (h \cdot a)$ gives
        \begin{equation}
        \begin{split}
            x \cdot (e_{sa}a) & = e_{sa}(x \cdot a) + (x \cdot e_{sa}) (h \cdot a)\\
            & = e_{sa}\left(\gamma_{ta}a-(\chi(h) \gamma_{sa} + \alpha(h) )(h \cdot a) + \sigma(a)\right)\\
            &\hspace{.15in}+(\gamma_{sa}e_{sa}-(\chi(h) \gamma_{sa} + \alpha(h) )e_{h \cdot sa})(h \cdot a)\\
            & = \gamma_{ta}a-(\chi(h) \gamma_{sa} + \alpha(h) )e_{sa}(h \cdot a) \\
            &\hspace{.15in} + \gamma_{sa}e_{sa}(h \cdot a) -(\chi(h) \gamma_{sa} + \alpha(h) )(h \cdot a) + \sigma(a)
        \end{split}
        \end{equation}
        If $sa= h \cdot sa$, then by \eqref{eq: relation of R coefficients Q_0} we have 
        $\chi(h) \gamma_{sa} + \alpha(h)=\gamma_{h\cdot sa}=\gamma_{sa}$.
        Now cancelling terms leaves $\gamma_{ta}a - (\chi(h) \gamma_{sa} + \alpha(h) )( h \cdot a) + \sigma(a)=x \cdot a$, showing that the relation $e_{sa}a = a$ is preserved by the action of $x$.
        When $sa \neq h \cdot sa$, we again find $x \cdot a= \gamma_{ta}a - (\chi(h) \gamma_{sa} + \alpha(h) )( h \cdot a) + \sigma(a)$.
        Similarly, the action of $x$ preserves the relation $a e_{ta}= a$. Hence this data gives an action.
        
        (b) Continuing the use of the notation and results stated in Proposition \ref{prop:H action on kQ}, every Hopf action of $R$ on $\fk Q$ is determined by (i), (ii), ($\sigma$1), and ($\sigma$2).
        In particular, $x \cdot a = \gamma_{ta}a-\gamma_{h \cdot sa}(h \cdot a) + \sigma(a)$ for all $a \in Q_1$ where $\sigma(a) = e_{sa} (x \cdot a) e_{h \cdot ta}$ and $\gamma_{h \cdot sa}=\chi(h)\gamma_{sa} + \alpha(h)$ as in Proposition \ref{prop:Action of R on Q_0}. 
       It remains to show ($\sigma$3) holds.

        Consider $\sigma(g \cdot a)$ for all $g \in G$.
        Using the relation in $R$ that $xg=\chi(g)gx+\alpha(g)(1-h)g$ for all $g \in G$, we obtain the following series of equalities:
        \begin{equation}
        \begin{split}
            \sigma(g \cdot a) & = e_{s(g \cdot a)}(x \cdot (g \cdot a)) e_{h \cdot t(g \cdot a)}\\
            &= e_{g \cdot sa}((xg) \cdot a) e_{hg \cdot ta}\\
            &= e_{g \cdot sa}((\chi(g)gx+\alpha(g)(1-h)g) \cdot a) e_{hg \cdot ta}\\
            &= e_{g \cdot sa}(\chi(g)g(x \cdot a) e_{hg \cdot ta}+e_{g \cdot sa}\alpha(g)(1-h)(g \cdot a) e_{hg \cdot ta}\\
            &= \chi(g)ge_{ sa}(x \cdot a) e_{h \cdot ta}+e_{g \cdot sa}\alpha(g)(1-h)(g \cdot a) e_{hg \cdot ta}\\
             &= \chi(g)g \sigma(a)+e_{g \cdot sa}(1-h)\alpha(g)(g \cdot a) e_{hg \cdot ta}\\
             &= \chi(g)g \sigma(a)+(e_{g \cdot sa}-e_{gh \cdot sa})\alpha(g)(g \cdot a) e_{hg \cdot ta}.
             \end{split}
        \end{equation}
        This shows ($\sigma$3) holds and thus proves that every Hopf action of $R$ on $\fk Q$ is of the form \eqref{eq:action of R on KQ}.
\end{proof}

In what remains, we will consider the actions of Hopf algebras that arise as Hopf-Ore extensions of group algebras using the general results proven above.

\section{Applications in small GK-dimension}\label{section:NHAOGKD1&2}

    In this final section, we apply our main result to some 
    affine Noetherian Hopf algebras, which can be thought of as quantum analogues to affine algebraic groups \cite{G12}.
    An important invariant of such algebras is their Gelfand-Kirillov dimension, which can be viewed as the analogue of Krull dimension in the noncommutative setting (see \cite{Bell15} for a survey).
    More specifically, for a finitely generated $\fk$-algebra $A$ of polynomial bounded growth, the \emph{Gelfand-Kirillov dimension} is defined as
    \begin{equation}
    \textrm{GK-dim}(A):= \limsup_{n \to \infty} \frac{\log \dim(V^n)}{\log n}
    \end{equation}
    where $V$ is a finite-dimensional vector space containing 1 that generates $A$ as a $\fk$-algebra.}
    
    Gelfand-Kirillov dimension zero is equivalent to being a finite-dimensional $\fk$-vector space, and thus the collection of Hopf algebras of GK-dimension zero is still very large.  
    Motivated by the fact that coordinate rings of connected algebraic groups are Noetherian integral domains, we restrict our attention to Noetherian prime Hopf algebras.
    With these additional conditions, the only Hopf $\fk$-algebra of GK-dimension zero is $\fk$ itself, 
    so we look to Noetherian prime Hopf algebras of Gelfand-Kirillov dimension one or two.
    See \cite{G12,BZ21} for excellent surveys of classification results of these algebras.
    It turns out that several classes of algebras which appear in these results are either Hopf-Ore extensions of group algebras, or quotients of such algebras.
    We focus on the cases of GK-dimension one and two separately below.

    \subsection{Gelfand-Kirillov Dimension One}
    Now we apply our main results to two families of Noetherian prime Hopf algebras of GK-dimension one.  The first are the generalizations of Taft algebras introduced in \cite{LWZ07}, and the second were introduced by Brown and Zhang \cite{BZ10} as a generalization of a family of Hopf algebras introduced by Liu \cite{Liu09}.  These families of algebras, and certain specializations, play an important role in classification results.

    \subsubsection{\texorpdfstring{$H(n,t,q)$}{Lg}}\label{H(n,t,q)}
        The first family of algebras we consider is $H(n,t,q)$, defined as follows.
        
    \begin{definition}{$H(n,t,q)$ \cite[Ex. 2.7]{LWZ07}} Let $n,m,t$ be integers and $q$ be an $n$th primitive root of unity. The Hopf algebra $H(n,t,q)$ is defined as the $\fk$-algebra generated by $x$ and $g$, subject to the relations 
    \begin{equation*}
    g^n=1 \textrm{ and } xg=q^mgx.
    \end{equation*}

        The coalgebra structure of $H(n,t,q)$ is defined by 
        \begin{equation*}
        \Delta(g)=g \otimes g \textrm{ and } \Delta(x)=g^t \otimes x + x \otimes 1.
        \end{equation*}
        \end{definition}

        \begin{remark}
            Because of the relations imposed on $H(n,t,q)$, it suffices to set $1 \leq t,m \leq n$.
        \end{remark}
        
    Recalling our notation for Hopf-Ore extensions of group algebras at the beginning of Section \ref{section: actions of HOE of GA} and taking into account Remark \ref{rmk:structure of comultiplication},
    if we let $G= \langle g \mid g^n=1\rangle$, $h=g^{-t}$, $\chi(g)=q^m$, and $\alpha(g)=0$, 
    we have $H(n,t,q) \simeq \fk G(\chi, h, \alpha)$.
    Using the fact that $\gamma_{g^t \cdot i}=\chi(g^t) \gamma_i=\chi(g)^t \gamma_i= (q^m)^t \gamma_i$, the following description of the actions of $H(n,t,q)$ on $\fk Q$ is a direct consequence of Proposition \ref{prop:Action of R on Q_0} and Theorem \ref{prop:Action of R on kQ}.

\begin{proposition} \label{prop:Action of H(n,t,q) on Q_0}
The (filtered) Hopf actions of $H(n,t,q)$ on $\fk Q$ are completely determined by the data of Proposition \ref{prop:Action of R on Q_0} and Theorem \ref{prop:Action of R on kQ} subject to the specific constraints:
    \begin{enumerate}[(1)]
        \item The collection of scalars $(\gamma_i \in \fk )_{i \in Q_0}$ satisfies $\gamma_{g \cdot i}= q^{m}\gamma_i$ for all $i \in Q_0$ and for all $g \in G$;
        \item The endomorphism $\sigma$ satisfies $\sigma(a)=e_{sa}\sigma(a) e_{g^{-t} \cdot ta}$ for all $a \in Q_1$;
            \item The endomorphism $\sigma$ satisfies $\sigma(g \cdot a) = q^m g \sigma(a)$ for all $a \in Q_1$ and for all $g \in G$. 
    \end{enumerate}
        
            With this data, the $x$-action is given by 
            \begin{align}
                x \cdot e_i &= \gamma_ie_i -q^{mt} \gamma_i e_{g^{-t} \cdot i}  \hspace{0.5cm} \textrm{ for all }i \in Q_0 \textrm{ and }\\ 
                x \cdot a &= \gamma_{ta}a- q^m \gamma_{sa}(g^{-t} \cdot a) + \sigma(a) \textrm{ for all }a \in Q_1. \label{action of H(n,t,q) on KQ} 
            \end{align}
\end{proposition}

    \subsubsection{Generalized Liu Algebra: \texorpdfstring{$B(n,w,q)$}{Lg}}
    Next, we consider generalized Liu algebras as constructed in \cite{BZ10}. These algebras are defined as follows.
    
    \begin{definition}{$B(n,w,q)$} \cite[Section 3.4]{BZ10}
    Let $n$ be an integer greater than 1, with $\gcd(n,p)=1$ if $\textrm{char}(\fk)=p>0$.
        Let $q$ be a primitive $n^{\rm th}$ root of unity.
        Then $B(n,w,q)$ is the Hopf algebra generated by $g^{\pm 1}, h^{\pm 1},$ and $x$, subject to the relations
         \begin{align*}
             xg&=gx, \hspace{1 cm} x^n=1-g^w=1-h^n,\\
             xh&=qhx, \hspace{1 cm} gh=hg, \textrm{ and } gg^{-1}=1.
        \end{align*}

         The coalgebra structure of $B(n,w,q)$ is uniquely determined by 
         \begin{equation}
         \Delta(g)=g \otimes g, \hspace{.15cm} \Delta(h)=h \otimes h, \textrm{ and } \Delta(x)=1 \otimes x + x \otimes h .
         \end{equation}
        \end{definition}
    
    We first note that $B(n,w,q)$ is a quotient of a Hopf-Ore extension $\tilde{H}$ of a group algebra generated by the two grouplike elements. 
    More specifically, if we let $G= \langle g, h \rangle$, $\chi(g)=1$, $\chi(h)=q$, and $\alpha(g)=\alpha(h)=0$, then we have $B(n,w,q) \simeq \fk G ( \chi , h, \alpha)$.
    The results of Section \ref{section: actions of HOE of GA} give the following proposition.

\begin{proposition}\label{prop:Action of tilde{H} on kQ}
The (filtered) Hopf actions of $\tilde{H}$ on $\fk Q$ are completely determined by the data of Proposition \ref{prop:Action of R on Q_0} and Theorem \ref{prop:Action of R on kQ} subject to the specific constraints:
    \begin{enumerate}[(1)]
        \item The collection of scalars $(\gamma_i \in \fk)_{i \in Q_0}$ satisfies $\gamma_{g \cdot i}= \gamma_i$ and $\gamma_{h \cdot i}= q\gamma_i$ for all $i \in Q_0$; \label{4.9.1}
        \item The endomorphism $\sigma$ satisfies $\sigma(a)=e_{sa}\sigma(a) e_{h \cdot ta}$ for all $a \in Q_1$; \label{4.9.2}
        \item The endomorphism $\sigma$ satisfies $\sigma(g \cdot a) = g \sigma(a)$ and $\sigma(h \cdot a) = q h \sigma(a)$ for all $a \in Q_1$. \label{4.9.3}
            \end{enumerate}
        
            With this data, the $x$-action is given by 
            \begin{align}
                x \cdot e_i &= \gamma_ie_i -q \gamma_i e_{h \cdot i}  \hspace{0.5cm} \textrm{ for all }i \in Q_0 \textrm{ and }\\ 
                x \cdot a &= \gamma_{ta}a- q \gamma_{sa}(h \cdot a) + \sigma(a) \textrm{ for all }a \in Q_1. \label{action of tilde{H} on KQ} 
            \end{align}
\end{proposition}

This allows for the parameterization of the actions of $B(n,w,q)$ on $\fk Q$ by determining when actions of $\tilde{H}$ factor through $B(n,w,q)$.
    \begin{corollary} \label{cor: factors through kQ0}
        An (filtered) action of $\tilde{H}$ on $\fk Q_0$ factors through $B(n,w,q)$ if and only if for all $i \in Q_0$ we have both:
        \begin{itemize}
        \item the scalar $\gamma_i$ is an $n^{th}$ root of unity when $h^n\cdot i \neq i$;
        \item $g^w \cdot i = h^n \cdot i$.
        \end{itemize}
    \end{corollary}
    \begin{proof}
        Assume we have an action that factors through $B(n,w,q)$.
        From the proof of Corollary 3.9 in \cite{KO21} we have
        \begin{equation}
        \begin{split}
            x^n \cdot e_i &= (x^n \cdot e_i)(h^n \cdot e_i) + e_i (x^n \cdot e_i)\\
            & \vdots\\
            & = \gamma^n_i (e_i - e_{h^n \cdot i}). \label{eq:nth power}
        \end{split}
        \end{equation}
        Moreover, from the relation $x^n = 1-h^n$, we get
        \begin{equation}
            x^n \cdot e_i = \gamma^n_i (e_i - e_{h^n \cdot i})=  e_i - e_{h^n \cdot i}
        \end{equation}

        If $h^n \cdot i \neq i$, then  $e_i - e_{h^n \cdot i}\neq 0$ and thus $\gamma^n_i =1$. Further, since $1-g^w = 1 - h^n$, we get $g^w \cdot i = h^n \cdot i$.
        The converse holds by tracing backwards through this reasoning.
    \end{proof}

The next corollary extends the one above to determine when an action of $\tilde{H}$ on $\fk Q$ factors through $B(n,w,q)$.

    \begin{corollary}
        Assume have a (filtered) action of $\tilde{H}$ on $\fk Q$ such that the action on $\fk Q_0$ factors through $B(n,w,q)$.  Then the action on $\fk Q$ factors through $B(n,w,q)$ if and only if for all $a \in Q_1$ we have both:
        \begin{itemize}
        \item $h^n\cdot a = g^w \cdot a$;
        \item $\sigma^n (a)= \gamma_{ta}' a - \gamma_{sa}' (h^n \cdot a)  \textrm{ where } \gamma_{i}' = \begin{cases}
           1-\gamma_{i}^n & \textrm{ if } i = h^n \cdot i\\
           -1 & \textrm{ if } i \neq h^n \cdot i
        \end{cases}$ \quad for $i \in \{sa, ta\}$.
        \end{itemize}
    \end{corollary}
    \begin{proof}
    We first set out to find an explicit expression for $x^n\cdot a$. 
    To do this, we perform preliminary calculations to show that
        \begin{equation}\label{eq:exnae}
            e_{sa}(x^n\cdot a)e_{h^n \cdot ta} = \gamma_{ta}^n a - \gamma_{sa}^n\, h^n \cdot a + \sigma^n( a)
        \end{equation}
    in order to use it later in the proof.
    Starting with the left side of the equality, we first notice that $n$ applications of \eqref{action of tilde{H} on KQ} and use of Proposition \ref{prop:Action of tilde{H} on kQ}\eqref{4.9.3} to move powers of $h$ inside $\sigma$
    will result in an expression for $x^n\cdot a$ which is a linear combination of terms $\sigma^i(h^j \cdot a)$ with $0 \leq i+j \leq n$.
    Such terms will survive left projection by $e_{sa}$ and right projection by $e_{h^n \cdot ta}$ exactly when both of the following hold:
    \begin{equation}
       sa= s\sigma^i(h^j \cdot a) = h^j \cdot sa \qquad \text{and} \qquad h^n\cdot ta = t\sigma^i(h^j \cdot a) = h^{i+j} \cdot ta,
    \end{equation}
    where the second equality in each line uses Proposition \ref{prop:Action of tilde{H} on kQ}\eqref{4.9.2}.

    If $sa = h^j \cdot sa$ for any $0 < j < n$, we have from Proposition \ref{prop:Action of tilde{H} on kQ}\eqref{4.9.1} that $\gamma_{h^j \cdot sa} = q^j \gamma_{sa}$, forcing $\gamma_{sa}=0$ since $q$ is a primitive $n^{\rm th}$ root of unity.
    We also see by considering repeated application of \eqref{action of tilde{H} on KQ} (possibly requiring repeated use of Proposition \ref{prop:Action of tilde{H} on kQ}\eqref{4.9.1}) that a positive power of $\gamma_{sa}$ appears in the coefficient of $\sigma^i(h^j \cdot a)$  if and only if at least one power of $h$ has acted in producing that term in the repeated expansion, which is if and only if $j > 0$. 
    Thus, the only possible terms $\sigma^i(h^j \cdot a)$ with nonzero coefficients in $e_{sa}(x^n\cdot a)e_{h^n \cdot ta}$ are from $(i,j)$ with $j=0$ or $j=n$.
    
    Similarly, if $h^n \cdot ta = h^{i+j} \cdot ta$ for any $0 < i+j < n$, we have from Proposition \ref{prop:Action of tilde{H} on kQ}\eqref{4.9.1} and the fact that $q$ is a primitive $n^{\rm th}$ root of unity that $\gamma_{ta}=0$.
    We also see by considering repeated application of \eqref{action of tilde{H} on KQ} (possibly requiring repeated use of Proposition \ref{prop:Action of tilde{H} on kQ}\eqref{4.9.1}) that a positive power of $\gamma_{ta}$ appears in the coefficient of $\sigma^i(h^j \cdot a)$  if and only if, in the substitution and expansion steps, at least one step did not apply a power of neither $h$ nor $\sigma$ in producing that term in the repeated expansion, which is if and only if $i+j < n$. This further narrows the only possible terms $\sigma^i(h^j \cdot a)$ with nonzero coefficients in $e_{sa}(x^n\cdot a)e_{h^n \cdot ta}$ to be those from $(i,j)$ with $i+j=0$ or $i+j=n$.

    Using both these cases, the only possibilities for $(i,j)$ such that the term containing $\sigma^i(h^j \cdot a) $ has a nonzero coefficient are $(0,0), (n,0), (0,n)$.  
    It is straightforward to see the coefficients of these terms are exactly as in the formula \eqref{eq:exnae} we set out to prove.  

    Now, we use this explicit description for $e_{sa}(x^n \cdot a)e_{h^n \cdot ta}$ to calculate $x^n \cdot a$.
    Since $\Delta^2 (x^n)= x^n \otimes h^n \otimes h^n + 1 \otimes x \otimes h^n + 1 \otimes 1 \otimes x^n$ and $e_{sa}ae_{ta}=a$, using substitution of \eqref{eq:nth power}, we have
    \begin{align}
    \begin{split}
        x^n \cdot a & = x^n \cdot (e_{sa}ae_{ta})\\
        & = (x^n \cdot e_{sa})(h^n \cdot a) e_{h^n \cdot ta} + e_{sa}(x^n \cdot a) e_{h^n \cdot ta} + e_{sa}a(x^n \cdot e_{ta})\\
        & = \gamma^n_{sa}(e_{sa}-e_{h^n \cdot sa})(h^n \cdot a) + e_{sa}(x^n \cdot a)e_{h^n \cdot ta} + a (\gamma_{ta}^n (e_{ta}-e_{h^n \cdot ta}))\\
        & = \gamma_{sa}^n e_{sa}(h^n \cdot a) - \gamma_{sa}^n(h^n \cdot a) + e_{sa}(x^n \cdot a) e_{h^n \cdot ta} + \gamma_{ta}^n a - \gamma_{ta}^nae_{h^n \cdot ta}. \label{this}
        \end{split}
    \end{align}
    
    Substituting \eqref{eq:exnae} into \eqref{this} reveals
    \begin{equation}
        x^n \cdot a  =\gamma_{sa}^n e_{sa}(h^n \cdot a) - \gamma_{sa}^n(h^n \cdot a) + \gamma_{ta}^n a- \gamma_{sa}^n (h^n \cdot a) + \sigma^n (a)+ \gamma_{ta}^n a - \gamma_{ta}^nae_{h^n \cdot ta}. \label{Work idea}
    \end{equation}

    Consider the following four cases:

    \textbf{Case 1:} Suppose $sa= h^n \cdot sa$ and $ta= h^n \cdot ta$. Then the first two terms and similarly the last two terms in \eqref{Work idea} add to zero leaving
    \begin{align}
        x^n \cdot a & =\gamma_{ta}^n a - \gamma_{sa}^n(h^n \cdot a) + \sigma^n(a).
    \end{align}
    
    \textbf{Case 2:} Suppose $sa= h^n \cdot sa$ and $ta \neq h^n \cdot ta$. This causes the first two terms in \eqref{Work idea} to add to zero and the last term becomes zero using $ae_{h^n \cdot ta}=0$. 
    This gives
    \begin{align}
    \begin{split}
        x^n \cdot a & = \gamma_{ta}^n a - \gamma_{sa}^n(h^n \cdot a) + \sigma^n(a) + \gamma_{ta}^n a\\
        & = 2\gamma_{ta}^n a - \gamma_{sa}^n(h^n \cdot a) + \sigma^n(a).
    \end{split}
    \end{align}

    \textbf{Case 3:} Suppose $sa \neq h^n \cdot sa$ and $ta = h^n \cdot ta$. Then the last two terms in \eqref{Work idea} add to zero and the first term becomes zero because $e_{sa}(h^n \cdot a)=0$. 
    We then find
    \begin{align}
    \begin{split}
        x^n \cdot a & = -\gamma_{sa}^n (h^n \cdot a) + \gamma_{ta}^n a - \gamma_{sa}^n(h^n \cdot a) + \sigma^n(a)\\
        & = \gamma_{ta}^n a - 2\gamma_{sa}^n(h^n \cdot a) + \sigma^n(a).
        \end{split}
    \end{align}

    \textbf{Case 4:} Suppose $sa \neq h^n \cdot sa$ and $ta \neq h^n \cdot ta$. This makes the first and last terms zero since $e_{sa}(h^n \cdot a)=0$ and $ae_{h^n \cdot ta}=0$ respectively.
    Therefore,
    \begin{align}
    \begin{split}
        x^n \cdot a & = -\gamma_{sa}^n (h^n \cdot a) + \gamma_{ta}^n a - \gamma_{sa}^n(h^n \cdot a) + \sigma^n(a) + \gamma_{ta}^n a\\
        & = 2\gamma_{ta}^n a - 2\gamma_{sa}^n(h^n \cdot a) + \sigma^n(a).
        \end{split}
    \end{align}

    To encompass all four cases, we write $x^n \cdot a$ as 
    \begin{equation}
        x^n \cdot a = c_{ta} \gamma_{ta}^n a - c_{sa} \gamma_{sa}^n(h^n \cdot a) + \sigma^n (a) \textrm{ where } c_i = \begin{cases}
           1 & \textrm{ if } i = h^n \cdot i\\
           2 & \textrm{ if } i \neq h^n \cdot i
        \end{cases} \textrm{ for } i \in \{sa,ta\}. \label{form of xn on a}
    \end{equation}

    Now, we assume the action of $\tilde{H}$ on $\fk Q$ factors through $B(n,w,q)$. From the relation $1-g^w=1-h^n$, we require $g^w \cdot a = h^n \cdot a$.
    Acting by the relation $x^n = 1- h^n$ on the left of $a \in Q_1$, we find $x^n \cdot a = a-(h^n \cdot a)$.
    Using \eqref{form of xn on a}, this is true exactly when 
    \begin{equation}
        \sigma^n (a)= (1- c_{ta} \gamma_{ta}^n) a -(1- c_{sa} \gamma_{sa}^n)(h^n \cdot a) + \sigma^n (a) \textrm{ where } c_i = \begin{cases}
           1 & \textrm{ if } i = h^n \cdot i\\
           2 & \textrm{ if } i \neq h^n \cdot i
        \end{cases} \label{first form of sigma}
    \end{equation}
    for $i \in \{sa,ta\}$. 
    From Corollary \ref{cor: factors through kQ0}, $\gamma_{i}^n=1$ when $i \neq h^n \cdot i$, allowing us to simplify \eqref{first form of sigma} to
    \begin{equation}\label{this take 2}
    \sigma^n (a)= \gamma_{ta}' a - \gamma_{sa}' (h^n \cdot a) \textrm{ where } \gamma_{i}' = \begin{cases}
           1-\gamma_{i}^n & \textrm{ if } i = h^n \cdot i\\
           -1 & \textrm{ if } i \neq h^n \cdot i
        \end{cases} \textrm{ for } i \in \{sa,ta\}.
    \end{equation}

    The converse holds by tracing backwards through this reasoning: the assumption that $g^w \cdot a = h^n \cdot a$ for all $a \in \fk Q_1$ gives us that the action preserves the relation $1-g^w=1-h^n$, and the second assumption gives us that the action preserves the relation $x^n = 1-h^n$ by working upwards from \eqref{this take 2}.
    \end{proof}

    \subsection{Gelfand-Kirillov Dimension Two}\label{sec:GK dim 2}
    Now we apply our main results to a family of Noetherian prime Hopf algebras of GK-dimension two introduced by Goodearl and Zhang \cite{GZ10} which also play an important role in classification results; see the surveys mentioned above.
    For a fixed $n$ and $q$, the algebra $C(n,q)$ is defined in the following manner:

    \begin{definition}{$C(n,q)$ \cite[Construction 1.4]{GZ10}}
        Given $n \in \Z$ and $q \in \fk^\times$, the $\fk$-algebra $C(n,q)$ is generated by $g^{\pm 1}$ and $x$, subject to the relation
        \begin{equation*}
            xg=qgx+g^n-g.
        \end{equation*}
        The unique Hopf algebra structure on $C(n,q)$ is determined by 
        \begin{equation*}
            \Delta(g)=g \otimes g \textrm{ and } \Delta(x)=1 \otimes x + x \otimes g^{n-1}.
        \end{equation*}
    \end{definition}
    
    Letting $G= \langle g \rangle,$ $\chi(g)=q$, $h=g^{n-1}$, and $\alpha$ be such that $\alpha(g^i)=-1-q-q^2-\cdots - q^{i-1}$,
    we get $C(n,q) \simeq \fk G(\chi, h, \alpha)$.
    We note that $C(n,q)$ is the first family of algebras in this paper of a Hopf-Ore extension of a group algebra for which $\alpha \not\equiv 0$, that is, the derivation  $\delta$ is nontrivial.
    Therefore, the action of $x$ on $\fk Q$ is given by the following proposition.

\begin{proposition}
The (filtered) Hopf action of $C(n,q)$ on $\fk Q$ is completely determined by the data of Proposition \ref{prop:Action of R on Q_0} and Theorem \ref{prop:Action of R on kQ} subject to the specific constraints
    \begin{enumerate}[(1)]
        \item The collection of scalars $(\gamma_i \in \fk)_{i \in Q_0}$ satisfies $\gamma_{g \cdot i}= q\gamma_i$ for all $i \in Q_0$ and for all $g \in G$; 
        \item The endomorphism $\sigma$ satisfies $\sigma(a)=e_{sa}\sigma(a) e_{g^{n-1} \cdot ta}$ for all $a \in Q_1$;
        \item The endomorphism $\sigma$ satisfies $\sigma(g \cdot a) = q g \sigma(a)- (e_{g \cdot sa} - e_{g^n \cdot sa})(g \cdot a)e_{g^n \cdot ta}$ for all $a \in Q_1$ and for all $g \in \langle g^{\pm 1} \rangle$. 
    \end{enumerate}
        
        With this data, the $x$-action is given by 
        \begin{align}
            x \cdot e_i &= \gamma_ie_i +(1 + q + q^{2} + \hdots + q^{(n-2)})\gamma_i e_{g^{n-1} \cdot i}  \hspace{0.5cm} \textrm{ for all }i \in Q_0 \textrm{ and }\\ 
            x \cdot a &= \gamma_{ta}a + (1 + q + q^{2} + \hdots + q^{(n-2)}) \gamma_{sa}(g^{n-1} \cdot a) + \sigma(a) \textrm{ for all  }a \in Q_1. \label{action of C(n,q) on KQ} 
        \end{align}

\end{proposition}

\bibliographystyle{alpha} 
\bibliography{HopfOre}

\end{document}

%% file: preamble-HopfOre.tex
\usepackage{amsmath, amsthm, amsfonts, amssymb, xy,  mathrsfs, graphicx, lscape, mathdots}
\usepackage[usenames, dvipsnames]{color}
\usepackage[margin=1in]{geometry}

\numberwithin{equation}{section}
\newtheorem{theorem}[equation]{Theorem}
\newtheorem*{thm}{Theorem}
\newtheorem{proposition}[equation]{Proposition}
\newtheorem{lemma}[equation]{Lemma}
\newtheorem{corollary}[equation]{Corollary}

\theoremstyle{definition}
\newtheorem{definition}[equation]{Definition}

\newtheorem{remark}[equation]{Remark}

\newtheorem{assumption}[equation]{Assumption}

\usepackage[colorlinks=true, pdfstartview=FitV, linkcolor=black, citecolor=black, urlcolor=black]{hyperref}

\newcommand{\Z}{\mathbb{Z}}

\newcommand{\fk}{\Bbbk}

\renewcommand{\phi}{\varphi}

\renewcommand{\tilde}[1]{\widetilde{#1}}

\newcommand{\comment}[1]{}

\makeatletter
\def\Ddots{\mathinner{\mkern1mu\raise\p@
\vbox{\kern7\p@\hbox{.}}\mkern2mu
\raise4\p@\hbox{.}\mkern2mu\raise7\p@\hbox{.}\mkern1mu}}
\makeatother